\documentclass[11pt, psamsfonts]{amsart}

\usepackage{amssymb,amsfonts}
\usepackage[all,arc]{xy}
\usepackage{enumerate}
\usepackage{mathrsfs}
\usepackage{thmtools}
\usepackage{thm-restate}
\usepackage{datetime}
\usepackage{hyperref}
\usepackage{graphicx}
\usepackage{cleveref}
\usepackage{comment}

\calclayout

\newtheorem{thm}{Theorem}[section]
\newtheorem{cor}[thm]{Corollary}
\newtheorem{prop}[thm]{Proposition}
\newtheorem{lem}[thm]{Lemma}
\newtheorem{conj}[thm]{Conjecture}
\newtheorem{ques}[thm]{Question}
\newtheorem{prob}[thm]{Problem}

\theoremstyle{definition}
\newtheorem{defn}[thm]{Definition}

\newtheorem{exmp}[thm]{Example}

\theoremstyle{remark}
\newtheorem{rem}[thm]{Remark}

\newtheorem{obs}[thm]{Observation}

\DeclareMathOperator{\cart}{\square}

\makeatletter
\let\c@equation\c@thm
\makeatother
\numberwithin{equation}{section}

\makeatletter
\newcommand*\bigcdot{\mathpalette\bigcdot@{.5}}
\newcommand*\bigcdot@[2]{\mathbin{\vcenter{\hbox{\scalebox{#2}{$\m@th#1\bullet$}}}}}
\makeatother

\makeatletter
\def\subsection{\@startsection{subsection}{3}%
  \z@{.5\linespacing\@plus.7\linespacing}{.1\linespacing}%
  {\bfseries}}
\makeatother

\newcommand{\customfootnotetext}[2]{{% Group to localize change to footnote
  \renewcommand{\thefootnote}{#1}% Update footnote counter representation
  \footnotetext[0]{#2}}}% Print footnote text

\newcommand{\N} { \mathbb{N}}

\newcommand{\Z} { \mathbb{Z}}

\newcommand{\cut} { \backslash}

\usepackage{etoolbox}
\makeatletter
\patchcmd{\@maketitle}
  {\ifx\@empty\@dedicatory}
  {\ifx\@empty\@date \else {\vskip3ex \centering\footnotesize\@date\par\vskip1ex}\fi
   \ifx\@empty\@dedicatory}
  {}{}
\patchcmd{\@adminfootnotes}
  {\ifx\@empty\@date\else \@footnotetext{\@setdate}\fi}
  {}{}{}
\makeatother
\DeclareMathOperator{\Cdot}{\bigcdot}

\title[Tightness of domination inequalities for direct product graphs]{Tightness of paired and upper domination inequalities for direct product graphs}
\author{Amanda Burcroff}

\address[]{Centre for Mathematical Sciences, University of Cambridge, Cambridge, UK}
\email{agb63@cam.ac.uk \textrm{ or } burcroff@umich.edu}

\begin{document}
\maketitle
\vspace{-1cm}
\begin{abstract}
    A set $D$ of vertices in a graph $G$ is called {\it dominating} if every vertex of $G$ is either in $D$ or adjacent to a vertex of $D$.  The {\it paired domination number} $\gamma_{\mathrm{pr}}(G)$ of $G$ is the minimum size of a dominating set whose induced subgraph admits a perfect matching, and the {\it upper domination number} $\Gamma(G)$ is the maximum size of a minimal dominating set.  In this paper, we investigate the sharpness of two multiplicative inequalities for these domination parameters, where the graph product is the direct product $\times$. 
    
    We show that for every positive constant $c$, there exist graphs $G$ and $H$ of arbitrarily large diameter such that $\gamma_{\mathrm{pr}}(G \times H) \leq c\gamma_{\mathrm{pr}}(G)\gamma_{\mathrm{pr}}(H)$, thus answering a question of Rall as well as two questions of Paulraja and Sampath Kumar. We then study when this inequality holds with $c = \frac{1}{2}$, in particular proving that it holds whenever $G$ and $H$ are trees. Finally, we demonstrate that the inequality $\Gamma(G \times H) \geq \Gamma(G) \Gamma(H)$, due to  Bre{\v s}ar, Klav{\v z}ar, and Rall, is tight.  
\end{abstract}
\vspace{-0.5cm}
\section{Introduction}
The interplay between domination parameters and graph products has been the subject of myriad studies, as surveyed in \cite[Chapter 28]{HIK} and \cite{NR}.  This most famously includes Vizing's conjecture involving the Cartesian product and the domination number \cite{Viz}.  In this paper, we focus on multiplicative inequalities involving the direct product and two variants of the domination number.

Let $G = (V(G),E(G))$ and $H = (V(H),E(H))$ be graphs with no isolated vertices.  The {\it direct product} of $G$ and $H$, denoted by $G \times H$, is the graph on $V(G) \times V(H)$ where vertices $(u_G, u_H)$ and $(v_G,v_H)$ are adjacent if and only if $\{u_G, v_G\} \in E(G)$ and  $\{u_H,v_H\} \in E(H)$.  A set $D \subseteq V(G)$ is called {\it dominating} if every vertex of $G$ is either in $D$ or adjacent to a vertex of $D$.

The first graph parameter we consider is the paired domination number, introduced by Haynes and Slater \cite{HS} in 1998.  The {\it paired domination number} $\gamma_{\mathrm{pr}}(G)$ of $G$ is the minimum size of a dominating set whose induced subgraph admits a perfect matching. Such a set is called a {\it paired dominating set}.  Haynes and Slater viewed paired domination as a model for assigning pairs of neighboring guards, acting as backups for each other, to adjacent vertices such that each vertex is ``watched by'', i.e., in the closed neighborhood of, a guard. 

 It is straightforward to show that the direct product of two paired dominating sets is a paired dominating set, hence $\gamma_{\mathrm{pr}}(G \times H) \leq \gamma_{\mathrm{pr}}(G)\gamma_{\mathrm{pr}}(H)$.  Bre{\v s}ar, Klav{\v z}ar, and Rall \cite{BKR} provided a sufficient condition for achieving equality involving packing numbers and the total domination number. They furthermore showed that the ratio of $\gamma_{\mathrm{pr}}(G \times H)$ to $\gamma_{\mathrm{pr}}(G)\gamma_{\mathrm{pr}}(H)$ can approach $\frac{1}{2}$ from above when $G$ and $H$ are both taken to be a subdivided star, i.e., a star graph where each edge is replaced by a path of length $2$ by adding a vertex of degree $2$.  They  proceeded to ask whether this ratio is greater than $\frac{1}{2}$ for all graphs $G$ and $H$.  In 2008, Rall \cite{Ral} generalized this question as follows.
 \begingroup
\renewcommand*{\theques}{\Alph{ques}}
\begin{ques}\emph{(\cite[Question 5]{Ral})}\label{ques: Rall}
Does there exist $c > 0$ such that $\gamma_{\mathrm{pr}}(G \times H) > c \gamma_{\mathrm{pr}}(G)\gamma_{\mathrm{pr}}(H)$ holds for all graphs $G$ and $H$?
\end{ques}
In 2010, Paulraja and Sampath Kumar \cite{PS} constructed some families of graphs $\{H_n\}_{n \in \N}$ with bounded paired domination number and diameter such that $\gamma_{\mathrm{pr}}(H_n \times H_n) = \frac{1}{2}\gamma_{\mathrm{pr}}(H_n)\gamma_{\mathrm{pr}}(H_n)$.  This answered the question of Bre{\v s}ar, Klav{\v z}ar, and Rall in the negative, but still left Question \ref{ques: Rall} open.  They furthermore asked the following related questions.

\begin{ques}\emph{(\cite[Problem 2.1]{PS})}\label{ques: PS1}
Do there exist graphs $G,H$ with diameter at least $4$ such that $\gamma_{\mathrm{pr}}(G \times H) \leq \frac{1}{2}\gamma_{\mathrm{pr}}(G)\gamma_{\mathrm{pr}}(H)$?
\end{ques}

\begin{ques}\emph{(\cite[Problem 2.2]{PS})}\label{ques: PS2}
Do there exist graphs $G,H$ and an integer $k \geq 8$ such that $\gamma_{\mathrm{pr}}(G), \gamma_{\mathrm{pr}}(H) \geq k$ and $\gamma_{\mathrm{pr}}(G \times H) \leq \frac{1}{2}\gamma_{\mathrm{pr}}(G)\gamma_{\mathrm{pr}}(H)$?
\end{ques}
\endgroup
\setcounter{thm}{0}

We answer Question \ref{ques: Rall} in the negative and give a positive answer to Questions \ref{ques: PS1} and \ref{ques: PS2} with the following result.  

\begin{thm}\label{thm: intro pair prod}
For every $c > 0$, there exists a connected graph $G$ of arbitrarily large diameter such that 
$$\gamma_{\mathrm{pr}}(G \times G) < c\gamma_{\mathrm{pr}}(G)^2\,.$$
\end{thm}

The graphs we construct consist of certain direct products of complete graphs, each with a path appended.  The direct products of complete graphs, along with the related unitary Cayley graph of $\Z/n\Z$, have generated significant recent interest for their extremal properties in the domination chain \cite{AD, Bur, DI, Vem}.  

We then proceed to study the conditions under which graphs $G$ and $H$ satisfy the inequality
$\gamma_{\mathrm{pr}}(G \times H) \geq \frac{1}{2}\gamma_{\mathrm{pr}}(G)\gamma_{\mathrm{pr}}(H)$.  We show that this holds whenever $G$ and $H$ are trees, and more generally whenever their $3$-packing number coincides with their paired domination number.  This allows us to show that any pair of graphs $G', H'$ can appear as induced subgraphs of $G,H$, respectively, satisfying this inequality.

Next, we shift our focus to another domination parameter, the upper domination number.  A dominating set is called {\it minimal} if none of its proper subsets are dominating.  The {\it upper domination number} $\Gamma(G)$ of $G$ is the maximum size of a minimal dominating set in $G$.  

In 2007, Bre{\v s}ar, Klav{\v z}ar, and Rall \cite{BKR} proved that the Vizing-like inequality
\begin{equation}\label{eq: upper dom}
    \Gamma(G \times H) \geq \Gamma(G)\Gamma(H)
\end{equation}
holds for arbitrary graphs $G$ and $H$.  Similar inequalities involving the Cartesian product are proven in \cite{Bre, Chi}.  While Bre{\v s}ar, Klav{\v z}ar, and Rall were not able demonstrate that the bound in \eqref{eq: upper dom} is optimal, they suggested that equality may be achieved for the family of $2 \times n$ rook graphs, i.e.,  $G = H = K_2 \cart K_n$.   We prove in Section \ref{Upper Dom} that this is indeed the case for $n \geq 71$, thus establishing that \eqref{eq: upper dom} is tight.

In Section \ref{sec: prelims}, we provide the necessary preliminaries.   Section \ref{Paired Dom} contains our results involving paired domination on direct product graphs, including the proof of Theorem \ref{thm: intro pair prod}.  In  Section  \ref{Upper Dom}, we consider the upper domination number and prove that \eqref{eq: upper dom} is tight.    We conclude with several open questions in Section \ref{sec: further}.

\section{Preliminaries}\label{sec: prelims}
 A graph $G$ is a set of vertices $V(G)$ along with a set of undirected edges $E(G)$, excluding loops.  For any $U \subseteq V(G)$, the subgraph of $G$ induced by $U$, denoted $G[U]$, is the graph with vertex set $U$ and whose edge set is precisely the edge set $E(G)$ restricted to $U \times U$.  We denote the complete graph on $n$ vertices by $K_n$.   A {\it perfect matching} in a graph $G$ is a collection of pairs of vertices such that the vertices in a pair are adjacent and every vertex is contained in exactly one pair.  Alternatively, a perfect matching can be viewed as a collection of edges such that every vertex is adjacent to an edge in the collection and no two edges share a vertex.  The {\it diameter} of a graph $G$ is the maximum distance between any two vertices of $G$.
 
 Let $N(v)$ denote the {\it open neighborhood} of a vertex $v$, that is, the set of all vertices adjacent to $v$.  Let $N[v]$ denote the {\it closed neighborhood} of a vertex $v$, which is the open neighborhood of $v$ along with $v$ itself.  For $S \subseteq V(G)$, let $N[S] = \bigcup_{v \in S} N[v]$.  A vertex $u$ is a {\it private neighbor} of a vertex $v \in S \subseteq V(G)$ (with respect to $S$) if $u \in N[v]$ and $u \notin N[S \cut \{v\}]$.  Note that a vertex can be its own private neighbor.  Observe that a set $S \subseteq V(G)$ is dominating if and only if $N[S] = V(G)$, and furthermore a dominating set $S$ is minimal if and only if every $v \in S$ has a private neighbor.  
 
 The {\it Cartesian product} (sometimes called the {\it box product}) of two graphs $G$ and $H$, denoted by $G \cart H$, is the graph on vertex set $V(G) \times V(H)$ with vertex $(u_G,u_H)$ adjacent to $(v_G,v_H)$ if and only if either $u_G = v_G$ and $\{u_H,v_H\} \in E(H)$, or $\{u_G,v_G\} \in E(G)$ and $u_H = v_H$. When we take the product of multiple graphs simultaneously, denoted by $\times_{i=1}^t G_i$, the graph product used is always the direct product.  This definition follows associatively from the definition of the direct product of two graphs.  Identifying the vertex set of $K_{n_i}$ with the set $\{0,\dots,n_i - 1\}$ for $n_i \in \N$ and fixing $0 \leq a_i,b_i < n_i$, we then have that the vertex $(a_1,\dots,a_t)$ is adjacent to $(b_1,\dots,b_t)$ in $\times_{i=1}^t K_{n_i}$ if and only if $a_i \neq b_i$ for all $i \in \{1,\dots,t\}$.
 
 For a positive integer $k$, Meir and Moon \cite{MM} defined a $k$-packing of $G$ to be a set $P \subseteq V(G)$ such that every pair of distinct vertices $u,v \in P$ have distance greater than $k$.  Thus the notion of a $2$-packing is equivalent to a classical packing, and a $1$-packing is precisely an independent set.  The {\it $k$-packing number}, denoted by $\rho_k(G)$, is the order of the largest $k$-packing of $G$.  We will be particularly interested in the $3$-packing number.
 
 The graph parameters we study are variants of the classical domination number.  The {\it domination number} $\gamma(G)$ of a graph $G$ is the minimum size of a dominating set in $G$.  Another variant we make use of is the {\it total domination number} $\gamma_{\mathrm{t}}(G)$ of $G$, which is the minimum size of a dominating set in $G$ whose induced subgraph has no isolated vertices.  Such a set is called a {\it total dominating set}.  It is straightforward to show that $\gamma(G) \leq \gamma_{\mathrm{t}}(G) \leq \gamma_{\mathrm{pr}}(G)$ and $\gamma(G) \leq \Gamma(G)$ for any graph $G$ without isolated vertices.  
 
 Note that the paired domination number is not well defined on graphs with isolated vertices.  Thus we assume implicitly that whenever we consider the paired domination number of a graph, the graph has no isolated vertices.

%%%%%%%%%%%%%%%%%%%%%%%%%%%%%%%%
\section{Paired Domination and Direct Products}\label{Paired Dom}
In this section, we study lower bounds on the paired domination number of the direct product of two graphs in terms of the paired domination number of each graph.  In particular, we examine these inequalities on certain direct products of complete graphs. In doing so, we answer a question of Rall \cite{Ral} in the negative by showing that the ratio of $\gamma_{\mathrm{pr}}(G \times H)$ to $\gamma_{\mathrm{pr}}(G)\gamma_{\mathrm{pr}}(H)$ can be arbitrarily small.  We furthermore show that the previous statement can hold on families of graphs with arbitrarily large diameter and paired domination number, thus resolving the two questions of Paulraja and Sampath Kumar in \cite{PS}.

Meki{\v s} \cite{Mek} demonstrated the tightness of the inequality $\gamma(G \times H) \geq \gamma(G) + \gamma(H) - 1$ by evaluating the domination number of a direct product of complete graphs with order larger than the number of factors. His results are summarized in the following lemma.

\begin{lem}\emph{(\cite[Corollary 2.2]{Mek})}\label{lem: Mek tight}
Let $G = \times_{i=1}^t K_{n_i}$, where $t \geq 3$ and $n_i \geq t + 1$ for all $i$.  Then 
$$\gamma(G) = t+1 = \gamma_{\mathrm{t}}(G)\,.$$
\end{lem}

Using a construction similar to that used by Meki{\v s}, we can compute the paired domination number of products of complete graphs under the conditions of Lemma \ref{lem: Mek tight}.

\begin{lem}\label{lem: match dom prod com}
Let $G = \times_{i=1}^t K_{n_i}$, where $t \geq 3$ and $n_i \geq t + 1$ for all $i$. Then 
$$\gamma_{\mathrm{pr}}(G) = \begin{cases} t + 1 & \text{ if $t$ is odd; } \\ t + 2 &\text{ if $t$ is even. } \end{cases}$$
\end{lem}
\begin{proof}
Let $D = \{(0,\dots,0), (1,\dots,1),\dots,(t,\dots,t)\}$, and observe that $D$ is a dominating set.  If $t$ is odd, then the subgraph induced by $D$ is a complete graph of even size, and thus admits a perfect matching.  Hence $\gamma_{\mathrm{pr}}(G) \leq t+1$. By Lemma \ref{lem: Mek tight}, we have $\gamma_{\mathrm{pr}}(G) \geq \gamma_{\mathrm{t}}(G) = t + 1$, and the desired equality follows.

If $t$ is even, then $\gamma_{\mathrm{t}}(G) = t+1$ is odd, hence $\gamma_{\mathrm{pr}}(G) \geq \gamma_{\mathrm{t}}(G) + 1 = t + 2$. Consider the dominating set $D' = D \cup \{(1,0,\dots,0)\}$ of size $t+2$.  The set $D'$ admits a perfect matching where vertex $(2i,\dots,2i)$ is paired with vertex $(2i + 1, \dots, 2i + 1)$ for $0 \leq i \leq \frac{t-2}{2}$, and vertex $(t,t,\dots,t)$ is paired with vertex $(1,0,\dots,0)$.  We can conclude in this case that $\gamma_{\mathrm{pr}}(G) = t+2$.
\end{proof}

This calculation can be used to show that $\gamma_{\mathrm{pr}}$ exhibits additive behavior on certain direct products of complete graphs.  However, a direct product of complete graphs is either disconnected or has diameter at most $3$. In order to prove that this additive behavior extends to connected graphs with arbitrarily large diameter, we first investigate how the addition of new vertex of degree one can affect the paired domination number of a direct product of graphs.

\begin{lem}\label{lem: add vertex}
Let $G$ and $H$ be connected graphs.  Let $G'$ be obtained from $G$ by adding a new vertex $v'$ and attaching it via a single edge to some $v \in V(G)$.  Then 
$$\gamma_{\mathrm{pr}}(G' \times H) \leq 2\left(\gamma_{\mathrm{pr}}(G \times H) + \gamma_{\mathrm{pr}}(H)\right)\,.$$
\end{lem}
\begin{proof}
Let $D$ be a paired dominating set of $G \times H$ and $D_H$ be a paired dominating set of $H$.  Then define
$$D' = D \cup (\{v\} \times D_H) \subseteq V(G' \times H)\,,$$
where $v$ is the unique vertex adjacent to $v'$.
We claim that $D'$ is a dominating set of $G' \times H$.  To see that $D'$ is dominating, note that we need only check that vertices of the form $\{v'\} \times V(H)$ are dominated, as the set $D$ dominates the remaining vertices.  Since $D_H$ is a paired dominating set of $V(H)$, hence a total dominating set of $V(H)$, and $v$ is adjacent to $v'$, the vertices in $\{v\} \times D_H$ dominate $\{v'\} \times V(H)$. Thus $D'$ is a dominating set of $G' \times H$ of size at most $\gamma_{\mathrm{pr}}(G \times H) + \gamma_{\mathrm{pr}}(H)$.  

It is well known that $\gamma_{\mathrm{pr}}(L) \leq 2\gamma(L)$ for any graph $L$ (see, for example, \cite[Theorem 9]{HS}).  Combining this with the inequality $\gamma(G' \times H) \leq \gamma_{\mathrm{pr}}(G \times H) + \gamma_{\mathrm{pr}}(H)$, we can conclude that the desired inequality holds.
\end{proof}

We now define an operation on vertex-transitive graphs (including the direct products of complete graphs) that appends a path to an arbitrary vertex, thus increasing the diameter.
\begin{defn}\label{defn: append path}
Suppose $G$ is a vertex-transitive graph and $\ell$ is a natural number. Let $G^{\,\Cdot \ell}$ be the graph obtained by joining a path on $\ell$ vertices to $G$ with a bridge (connected to a vertex of degree $1$ on the path). We furthermore define $G^{\,\Cdot 0} = G$.
\end{defn}

\begin{figure}[h]
\begin{center}
\includegraphics[width=.9\linewidth]{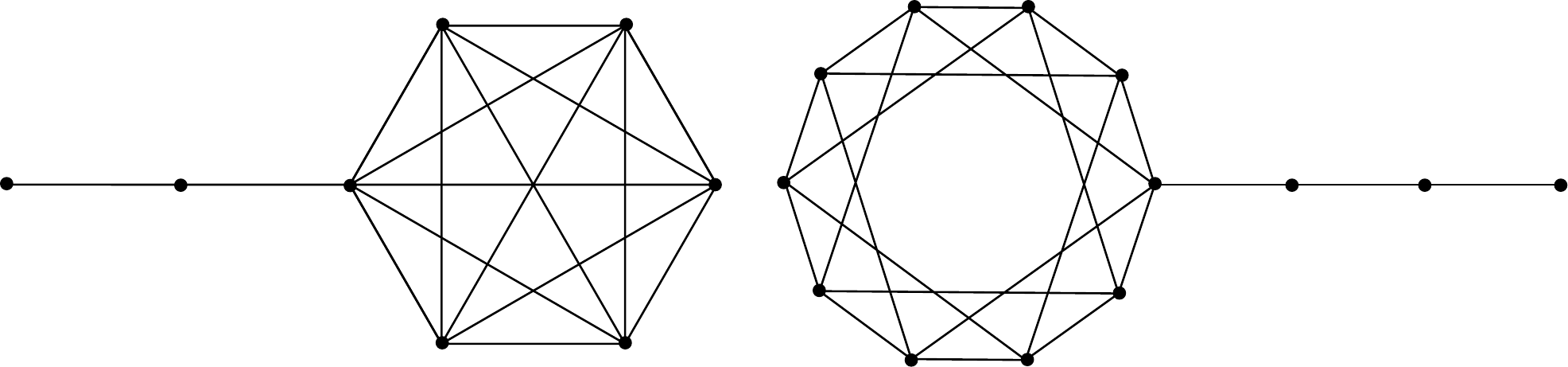}
\caption{Illustrating the construction of Definition \ref{defn: append path}, this figure depicts the graph $(K_6)^{\,\Cdot 2}$ on the left and $(K_2 \times K_5)^{\,\Cdot 3}$ on the right.\textsuperscript{$\dagger$}}
\label{Fig1}
\end{center}  
\end{figure}
\customfootnotetext{$\dagger$}{The graph $(K_m)^{\,\Cdot n}$ is the well-studied \emph{$(m,n)$-lollipop graph}. If $n_1,\dots,n_t$ are distinct primes and $n = n_1\cdots n_t$, then the graph $\times_{i=1}^t K_{n_i}$ is the \emph{unitary Cayley graph of $\Z/n\Z$}.  Thus graphs of the form $\left(\times_{i=1}^t K_{n_i}\right)^{\,\Cdot \ell}$ where the $n_i$ are distinct primes may be referred to as \emph{Cayleypops}.}

\begin{rem}
Note that the vertex-transitivity of $G$ guarantees that this construction is well-defined up to graph isomorphism. This operation may also be viewed as a particular {\it coalescence} of the graph $G$ with a path on $\ell + 1$ vertices.
\end{rem}

\begin{obs}\label{obs: pr not dec}
For any graph $G$ and $\ell \in \N$, the paired domination number of $G^{\,\Cdot \ell}$ is at least that of $G$.  Given a paired dominating set $P$ of $G^{\,\Cdot \ell}$, let $P'$ be its restriction to $G$.  Then $P'$ dominates every vertex of $G$ except at most one, where the path was attached.  It can be checked by casework that if one or two vertices need to be added to $P'$ in order to obtain a paired dominating set of $G$, then at least one or two vertices, respectively, of $P$ were contained in the path of length $\ell$.  Thus this yields a paired dominating set of $G$ of size at most $|P|$.
\end{obs}

\begin{lem}\label{lem: prod pops}
Fix $t \geq 3$ and $n_i \geq 2t + 1$.  For nonnegative integers $a$ and $b$, we have 
$$\gamma_{\mathrm{pr}}\left(\left(\times_{i=1}^t K_{n_i}\right)^{\Cdot a} \times \left(\times_{i=1}^t K_{n_i}\right)^{\Cdot b}\right) \leq  2^{a + b}\left((a + 2)t + 2a + 2\right) + 2^bb(t + a + 2)\,.$$
\end{lem}
\begin{proof}
Let $G = \times_{i=1}^t K_{n_i}$.  Note that Lemma \ref{lem: match dom prod com} implies $\gamma_{\mathrm{pr}}(G) \leq t + 2$ and $\gamma_{\mathrm{pr}}(G \times G) = 2t + 2$. Since $G^{\,\Cdot \ell}$ is obtained from $G$ by adding a path of length $\ell$, we have $\gamma_{\mathrm{pr}}(G^{\,\Cdot \ell}) \leq t + \ell + 2$.  Repeated applications of Lemma \ref{lem: add vertex} show that 
$$\gamma_{\mathrm{pr}}(G^{\,\Cdot \ell} \times H) \leq 2^\ell\left(\gamma_{\mathrm{pr}}(G \times H) + \ell\cdot\gamma_{\mathrm{pr}}(H)\right)\,.$$

In particular, we have
$$\gamma_{\mathrm{pr}}(G^{\,\Cdot a} \times G) \leq 2^a\left(\gamma_{\mathrm{pr}}(G \times G) + a\cdot\gamma_{\mathrm{pr}}(G)\right) \leq 2^a\left((a + 2)t + 2a + 2\right)\,.$$
Using this calculation, we can similarly bound
\begin{align*}
    \gamma_{\mathrm{pr}}(G^{\,\Cdot a} \times G^{\,\Cdot b}) &=  \gamma_{\mathrm{pr}}(G^{\,\Cdot b} \times G^{\,\Cdot a})\\
    &\leq 2^{b}\left(\gamma_{\mathrm{pr}}(G \times G^{\,\Cdot a}) + b\cdot \gamma_{\mathrm{pr}}(G^{\,\Cdot a})\right)\\
    &\leq 2^{b}\left(2^a\left((a + 2)t + 2a + 2\right) + b(t + a + 2)\right)\\
    &= 2^{a + b}\left((a + 2)t + 2a + 2\right) + 2^bb(t + a + 2)
\end{align*}
as desired.
\end{proof}

Using this calculation along with previous lemmas, we now proceed to prove Theorem \ref{thm: intro pair prod}.  

\begin{proof}[Proof of Theorem \ref{thm: intro pair prod}]
Fix $d,t,n_1,\dots,n_t \in \N$ such that $d < 2t < n_1,\dots,n_t$.  Consider the graph $\left(\times_{i=1}^t K_{n_i}\right)^{\Cdot d}$.  By Lemma \ref{lem: prod pops} we have
\begin{align*}
    \gamma_{\mathrm{pr}}\left(\left(\times_{i=1}^t K_{n_i}\right)^{\Cdot d} \times \left(\times_{i=1}^t K_{n_i}\right)^{\Cdot d}\right)&\leq  2^{2d}\left((d + 2)t + 2d + 2\right) + 2^dd(t + d + 2)\\
    &=\left(4^d(d + 2) + 2^dd\right)t + 4^d(2d + 2) + 2^dd(d + 2)   \,.
\end{align*}
On the other hand, Observation \ref{obs: pr not dec} and Lemma \ref{lem: prod pops} imply 
$$\gamma_{\mathrm{pr}}\left(\left(\times_{i=1}^t K_{n_i}\right)^{\Cdot d}\right) \geq \gamma_{\mathrm{pr}}\left(\times_{i=1}^t K_{n_i}\right)\geq t + 1\,.$$  Hence $$\gamma_{\mathrm{pr}}\left(\left(\times_{i=1}^t K_{n_i}\right)^{\Cdot d}\right)^2 \geq (t + 1)^2\,.$$

Note that $  \gamma_{\mathrm{pr}}\left(\left(\times_{i=1}^t K_{n_i}\right)^{\Cdot d} \times \left(\times_{i=1}^t K_{n_i}\right)^{\Cdot d}\right)$ is linear in $t$ while $\gamma_{\mathrm{pr}}\left(\left(\times_{i=1}^t K_{n_i}\right)^{\Cdot d}\right)^2$ is a quadratic in $t$ with positive coefficients.  Thus for any $c > 0$, $d \in \N$, and sufficiently large $t$, choosing $n_1,\dots,n_t > 2t$, we have 
$$\frac{\gamma_{\mathrm{pr}}\left(\left(\times_{i=1}^t K_{n_i}\right)^{\Cdot d} \times \left(\times_{i=1}^t K_{n_i}\right)^{\Cdot d}\right)}{\gamma_{\mathrm{pr}}\left(\left(\times_{i=1}^t K_{n_i}\right)^{\Cdot d}\right)^2} < c\,.$$
Clearly by construction, the diameter of $\gamma_{\mathrm{pr}}\left(\left(\times_{i=1}^t K_{n_i}\right)^{\Cdot d}\right)$ is at least $d$.  Therefore taking $G = \left(\times_{i=1}^t K_{n_i}\right)^{\Cdot d}$ for arbitrarily large $d$ and sufficiently large $t$ (depending on $c,d$), we have the desired inequality.
\end{proof}

\begin{rem}
Note that in the above construction, the integers $n_1,\dots,n_t$ can always be chosen to be distinct primes.  In this case, the graph $G$ a unitary Cayley graph of $\Z/m\Z$, where $m = n_1 \cdots n_t$, with a path on $d$ vertices appended to a vertex via a bridge, i.e., a Cayleypop.
\end{rem}

Though the inequality $\gamma_{\mathrm{pr}}(G \times H) \geq \frac{1}{2}\gamma_{\mathrm{pr}}(G)\gamma_{\mathrm{pr}}(H)$ fails for general graphs $G$ and $H$, we can investigate those graphs for which this inequality holds.  We now proceed to prove a new sufficient condition for this inequality to hold involving the $3$-packing number.  This sufficient condition then allows us to show that any pair of graphs can appear as subgraphs of a pair of graphs satisfying the aforementioned inequality.  Moreover, using a result of Bre{\v s}ar, Henning, and Rall, this condition also implies that this inequality holds whenever $G$ and $H$ are trees.

\begin{lem}\emph{(\cite[Theorem 3.5]{BHR})}\label{lem: pr vs rho3}
For any graph $G$, we have $\gamma_{\mathrm{pr}}(G) \geq 2\rho_3(G)$.  
\end{lem}

\begin{obs}\label{obs: packing product}
For any graphs $G$ and $H$, we have $\rho_3(G \times H) \geq \rho_3(G)\rho_3(H)$.  This follows by verifying that taking the direct product of $3$-packings in $G$ and $H$ yields a $3$-packing in $G \times H$.  
\end{obs}

\begin{prop}\label{meet 1/2 pr dom bound}
Let $G$ and $H$ be graphs satisfying $\gamma_{\mathrm{pr}}(G) = 2\rho_3(G)$ and $ \gamma_{\mathrm{pr}}(H) = 2\rho_3(H)$.  Then
$$\gamma_{\mathrm{pr}}(G \times H) \geq \frac{1}{2}\gamma_{\mathrm{pr}}(G)\gamma_{\mathrm{pr}}(H)\,.$$
\end{prop}
\begin{proof}
By Lemma \ref{lem: pr vs rho3} and Observation \ref{obs: packing product}, we have
\begin{align*}
    \gamma_{\mathrm{pr}}(G \times H) &\geq 2\rho_3(G \times H)\\
    &\geq 2\rho_3(G)\rho_3(H)\\
    &= 2\left(\frac{1}{2}\gamma_{\mathrm{pr}}(G)\right)\left(\frac{1}{2}\gamma_{\mathrm{pr}}(H)\right)\\
    &= \frac{1}{2}\gamma_{\mathrm{pr}}(G)\gamma_{\mathrm{pr}}(H)\,.\qedhere
\end{align*}
\end{proof}

In particular, Bre{\v s}ar, Henning, and Rall \cite{BHR} showed that the conditions of Proposition \ref{meet 1/2 pr dom bound} hold for trees.

\begin{thm}\emph{(\cite{BHR})}\label{thm: tree rho}
If $T$ is a tree, then $\gamma_{\mathrm{pr}}(T) = 2\rho_3(T)$.  
\end{thm}

Combining Proposition \ref{meet 1/2 pr dom bound} with Theorem \ref{thm: tree rho}, we immediately obtain the following result.

\begin{cor}\label{cor: tree prod}
The inequality $$\gamma_{\mathrm{pr}}(T_1 \times T_2) \geq \frac{1}{2}\gamma_{\mathrm{pr}}(T_1)\gamma_{\mathrm{pr}}(T_2)$$
holds whenever $T_1$ and $T_2$ are trees.
\end{cor}

 Bre{\v s}ar, Klav{\v z}ar, and Rall \cite{BKR} examined the subdivided star $S_n$ to show that the ratio of $\gamma_{\mathrm{pr}}(G \times G)$ to $\gamma_{\mathrm{pr}}(G)^2$ can be arbitrarily close to $\frac{1}{2}$.  Thus the constant factor in the inequality of Corollary \ref{cor: tree prod} is tight.  We now construct infinitely many graphs satisfying $\gamma_{\mathrm{pr}}(G \times H) \geq \frac{1}{2}\gamma_{\mathrm{pr}}(G)\gamma_{\mathrm{pr}}(H)$, with the added property that any graph can occur as a subgraph of such $G$ and $H$.

\begin{thm}\label{thm: prod subgraph}
For any graphs $G$ and $H$, there exist graphs $G'$ and $H'$ containing $G$ and $H$, respectively, as induced subgraphs such that 
$$\gamma_{\mathrm{pr}}(G' \times H') \geq \frac{1}{2}\gamma_{\mathrm{pr}}(G')\gamma_{\mathrm{pr}}(H')\,.$$
\end{thm}
\begin{proof}
Let $G'$ and $H'$ be the graphs formed by attaching a path on $2$ vertices, i.e., a copy of $K_2$, to each vertex of $G$ and $H$, respectively, with a bridge. By construction, $G'$ contains $G$ as an induced subgraph and $H'$ contains $H$ as an induced subgraph.  By Proposition \ref{meet 1/2 pr dom bound}, it is enough to show that $\gamma_{\mathrm{pr}}(G') = 2\rho_3(G')$ and $\gamma_{\mathrm{pr}}(H') = 2\rho_3(H')$.  We show that this is indeed the case.  

Let $|V(G)| = n$.  There are precisely $n$ new vertices of degree one in $G'$; label these by $x_1,\dots,x_n$.  Let $y_i$ be the unique vertex adjacent to $x_i$ for $1 \leq i \leq n$.  Note that the $x_i$ have distance at least $5$ from each other, so the $x_i$ form a $4$-packing of $G'$.  Hence we have $\rho_3(G') \geq \rho_4(G') \geq n$.   Moreover, observe that 
$$\{x_i : 1 \leq i \leq n\} \cup \{y_i : 1 \leq i \leq n\}$$
is a paired dominating set of $G'$, where $x_i$ is paired with $y_i$ for each $i$.  Thus $\gamma_{\mathrm{pr}}(G') \leq 2n$.  Since we have $\gamma_{\mathrm{pr}}(G') \geq 2\rho_3(G')$ by Lemma \ref{lem: pr vs rho3}, we can conclude $\gamma_{\mathrm{pr}}(G') = 2\rho_3(G') = n$.
\end{proof}
\begin{figure}[h]
\begin{center}
\includegraphics[width=.9\linewidth]{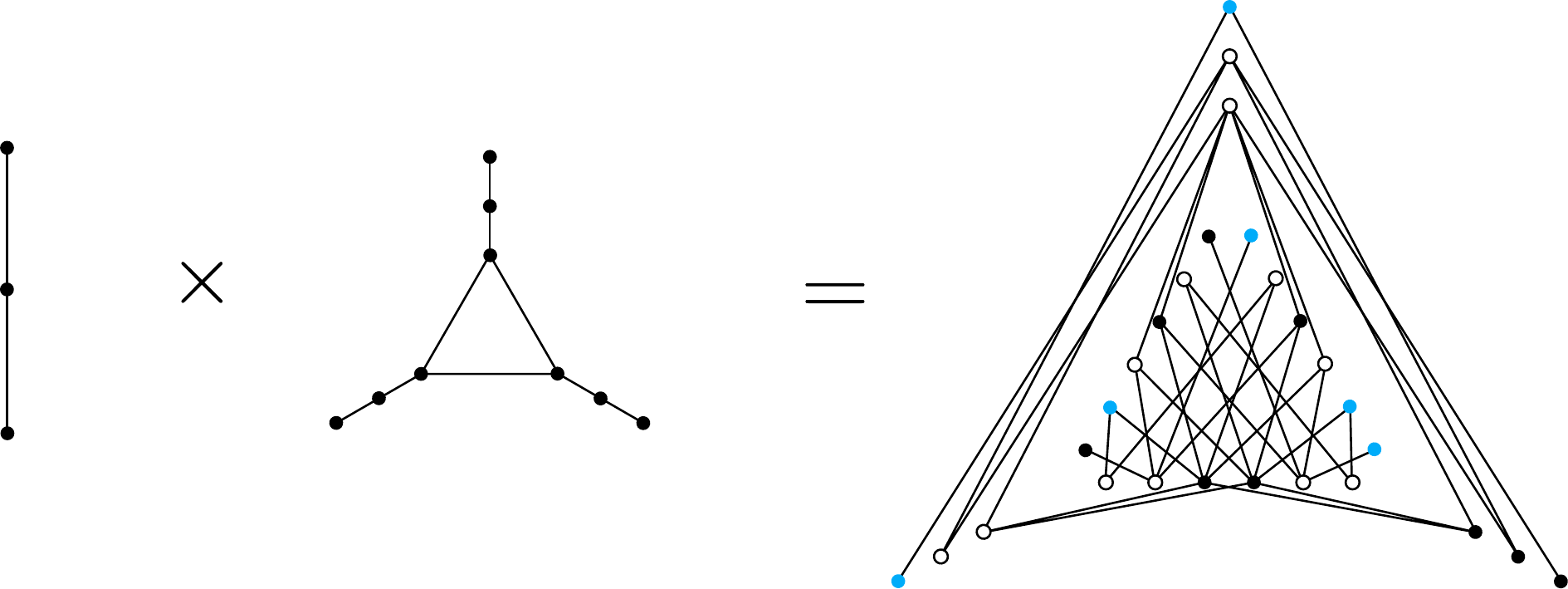}
\caption{The left and center graphs depict $G'$ and $H'$ constructed, as in the proof of Theorem \ref{thm: prod subgraph}, by attaching a path on two vertices to every vertex of $G = K_1$ and $H = K_3$, respectively, with a bridge.  The graph on the right is the direct product graph $G' \times H'$, with a $3$-packing of size $6$ shown in blue and a paired matching of size $12$ shown in white.}
\label{Fig2}
\end{center}  
\end{figure}
\begin{exmp}
Let $G$ be an isolated vertex and $H = K_3$ be a triangle. Let $G'$ and $H'$ be the graphs constructed as in the proof of \ref{thm: prod subgraph}, shown in Figure \ref{Fig2}.  

It can be easily verified that $\gamma_{\mathrm{pr}}(G') = 2$ and $\gamma_{\mathrm{pr}}(H') = 6$.  There is a $3$-packing of size $6$ in $G' \times H'$, as exhibited in Figure \ref{Fig2}.  Hence by Lemma \ref{lem: pr vs rho3} we have
$$\gamma_{\mathrm{pr}}(G' \times H') \geq 2\rho_3(G' \times H') \geq 12 \geq \frac{1}{2} \gamma_{\mathrm{pr}}(G') \gamma_{\mathrm{pr}}(H')\,,$$
as guaranteed by Theorem \ref{thm: prod subgraph}.
Furthermore, since $G' \times H'$ has a paired dominating set of size $12$, we can conclude $\gamma(G' \times H') = 12 = \gamma_{\mathrm{pr}}(G') \gamma_{\mathrm{pr}}(H')$.
\end{exmp}

\section{The Upper Domination Numbers of Direct Product Graphs}\label{Upper Dom}
In this section, we prove that the inequality 
\begin{equation}\label{eq: Gamma}
    \Gamma(G \times H) \geq \Gamma(G)\Gamma(H)
\end{equation}
is tight on a family of graphs with arbitrarily large upper domination number.  This supermultiplicative inequality was originally proven for any graphs $G, H$ in 2007 by Bre{\v s}ar, Klav{\v z}ar, and Rall \cite{BKR}.  While they were not able to prove the tightness of this inequality, they suggested that the $2 \times n$ rook graphs could be a possible candidate for attaining equality.  We show that this is indeed the case for $n \geq 71$.  

In general, the {\it $m \times n$ rook graph} is the graph $K_m \cart K_n$.  It can be viewed as the connectivity graph of a rook chess piece on an $m \times n$ chessboard.  We are particularly interested in the case $m = 2$, thus we denote the $2 \times n$ rook graph by $G_n$.  The graph $G_n$ can be viewed as a disjoint union of two complete graphs on $n$ vertices with a set of $n$ edges forming a perfect matching between them. It is straightforward to see that the upper domination number of $G_n$ is $n$, where a subgraph isomorphic to $K_n$ yields a minimal dominating set of maximum size.

Bre{\v s}ar, Klav{\v z}ar, and Rall noted that if $\Gamma(G \times G) = \Gamma(G)^2$ for some graph $G$, then necessarily $\alpha(G) < \Gamma(G)$.  Observe that for every graph $G$, the weak inequality $\alpha(G) \leq \Gamma(G)$ holds.  Since $\alpha(G_n) = 2 < n = \Gamma(G_n)$, these parameters can have an arbitrarily large difference on this family.  This provides some motivation for studying the tightness of \eqref{eq: Gamma} on the $2 \times n$ rook graphs.

\begin{thm}\label{thm: upper dom}
Fix $n \geq 71$ and let $G_n = K_2 \cart K_n$.  Then $\Gamma(G_n \times G_n) = n^2 = \Gamma(G_n)^2$.
\end{thm}

Before presenting the proof of this theorem, we examine the structure of the graph $G_n \times G_n$.

There is a natural identification of the vertices of $G_n$ with pairs in $\Z_2 \times \Z_n$, where a vertex $(a,b)$ adjacent to $(a',b')$ if and only if exactly one of $a = a'$ or $b = b'$ holds.  This induces an identification of the vertices of $G_n \times G_n$ with the $4$-tuples $(a,b,c,d) \in \Z_2 \times \Z_n \times \Z_2 \times \Z_n$.  We then partition the vertices of $G_n \times G_n$ into four classes of size $n^2$, namely 
$$N_{ij} = \{i\} \times \Z_n \times \{j\} \times \Z_n \;\;\;\text{ for } i,j \in \{0,1\}\,.$$

Given a dominating set $D$ of $G_n \times G_n$, we accordingly partition $D$ into four (possibly empty) classes, $D_{ij} = D \cap N_{ij}$.  Our proof of Theorem \ref{thm: upper dom} will proceed via casework on which classes $D_{ij}$ are nonempty.  If we assume $|D| > n^2$, then at least two classes $D_{ij}$ must be nonempty.  By symmetry on the labeling of the vertices, it is enough to consider four cases: 
\begin{enumerate}[(i)]
    \item\label{a} The classes $D_{00}$ and $D_{11}$ are nonempty and the remaining classes are empty.
    \item\label{b} The classes $D_{00}$ and $D_{01}$ are nonempty and the remaining classes are empty.
    \item\label{c} The classes $D_{00}$, $D_{10}$, and $D_{01}$ are nonempty and $D_{11}$ is empty.
    \item\label{d} All four of the classes $D_{ij}$ are nonempty.
\end{enumerate}

Cases (\ref{a}), (\ref{b}), and (\ref{d}) are relatively straightforward.  Case (\ref{c}) is the most complex, so we address several lemmas corresponding to this case before presenting the proof of Theorem \ref{thm: upper dom}.  To this end, we introduce some additional notation describing the structure of $G_n \times G_n$.  The {\it $b_0$-column} of $N_{ij}$ is the set of vertices $\{(i,b_0,j,d) : d \in \Z_n\}$ for fixed $b_0 \in \Z_n$.  The {\it $d_0$-row} of $N_{ij}$ is analogously the set of vertices $\{(i,b,j,d_0) : b \in \Z_n\}$ for fixed $d_0 \in \Z_n$.  A {\it row} (resp. {\it column}) of $D_{ij}$ is the intersection of $D$ with a row (resp. column) of $N_{ij}$.  Suppose $(i,j) \neq (i',j') \in \Z_2 \times \Z_2$, and fix $u \in N_{ij}$, $v \in N_{i'j'}$. The vertices $u$ and $v$ lie in {\it corresponding columns} if $u$ lies in the $b_0$-column of $N_{ij}$ and $v$ lies in the $b_0$-column of $N_{i'j'}$ for some $b_0 \in \Z_n$.  Similarly, $u$ and $v$ lie in {\it corresponding rows} if $u$ lies in the $d_0$-row of $N_{ij}$ and $v$ lies in the $d_0$-row of $N_{i'j'}$ for some $d_0 \in \Z_n$.  We say that $u$ and $v$ are {\it corresponding vertices} if they lie in corresponding rows and corresponding columns.

We now state a few basic adjacency properties in $G_n \times G_n$ using the additional language of corresponding rows, columns, and vertices.  These properties are illustrated in Figure \ref{Fig3} for the case $n = 3$.
\begin{obs}\label{obs: adj}
Consider the graph $G_n \times G_n$.  For $i,j \in \Z_2$, we have that
\begin{enumerate}[(a)]
    \item Two vertices in $N_{ij}$ are adjacent if and only if they are in different rows and columns.  Hence each $N_{ij}$ induces a $K_n \times K_n$ subgraph.
    \item A vertex $u \in N_{i0}$ is adjacent to $v \in N_{i1}$ if and only if $u$ and $v$ lie in corresponding rows and are not corresponding vertices.  Thus corresponding rows in $N_{i0}$ and $N_{i1}$ induce a $K_2 \times K_n$ subgraph.  
    \item A vertex $u \in N_{0j}$ is adjacent to $v \in N_{1j}$ if and only if $u$ and $v$ lie in corresponding columns and are not corresponding vertices.  Thus corresponding columns in $N_{0j}$ and $N_{1j}$ induce a $K_n \times K_2$ subgraph.  
    \item Fix $j' \neq j \in \Z_2$.  Then $u \in N_{0j}$ is adjacent to $v \in N_{1j'}$ if and only if $u$ and $v$ are corresponding vertices.
\end{enumerate}
\end{obs}

\begin{figure}[h]
\begin{center}
\includegraphics[width=.8\linewidth]{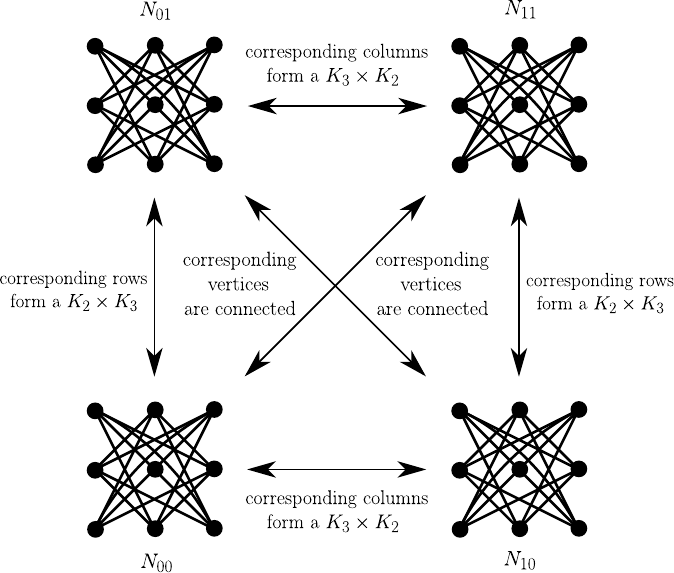}
\caption{An illustration of the graph $G_3 \times G_3$ with the vertices partitioned into four classes $N_{ij}$ for $i,j \in \Z_2$, connected as shown.}
\label{Fig3}
\end{center}  
\end{figure}

\begin{lem}\label{lem: c1}
Suppose $D$ is a minimal dominating set of $G_n \times G_n$ of size greater than $n^2$ and the conditions of Case (\ref{c}) hold. Then at least $n^2 - 10n$ members of $D$ are contained in $D_{00}$ and have a private neighbor in $N_{11}$.  Moreover, no vertex of $N_{00}$ is a private neighbor of a vertex in $D$.
\end{lem}
\begin{proof}
Note that the closed neighborhood of each vertex in $N_{ij}$ contains $n^2 - 2n + 1$ vertices of $N_{ij}$.  Since $D_{00}$, $D_{01}$, and $D_{10}$ are all nonempty, then there are at most $6n$ vertices of $D$ with a private neighbor in $N_{00} \cup N_{10} \cup N_{01}$.  By Observation \ref{obs: adj}, edges between $N_{10}$ and $N_{11}$ must connect vertices in corresponding rows, which induce a $K_2 \times K_n$ subgraph.  Thus the neighborhood of any two vertices in the same row of $N_{10}$ will contain all vertices in the corresponding row of $N_{11}$. Hence in each row of $D_{10}$, there are at most $2$ vertices with a private neighbor in $N_{11}$, and similarly for columns of $D_{01}$.  So there are at most $4n$ vertices of $D_{10} \cup D_{01}$ with a private neighbor in $N_{11}$.  The remaining vertices of $D$, of which there are at least $n^2 - 10n$, must be members of $D_{00}$ with a private neighbor in $N_{11}$.

It remains to prove the second claim.  By Observation \ref{obs: adj}, each vertex of $N_{00}$ is adjacent to all but $2n-1$ vertices of $N_{00}$.  Using our lower bound on the size of $D_{00}$ and our assumption that $n \geq 71$, we have $|D_{00}| \geq n^2 - 10n > 2n + 1$.  Thus at least two vertices of $D_{00}$ are adjacent to each vertex of $N_{00}$, so no vertex in $N_{00}$ is a private neighbor of a vertex in $D$.
\end{proof}

\begin{lem}\label{lem: c3}
Suppose $D$ is a minimal dominating set of $G_n \times G_n$ of size greater than $n^2$ and the conditions of Case (\ref{c}) hold.  Then
$$|D \cut D_{00}| = |D_{10}| + |D_{01}| > 2n - 43\,.$$
\end{lem}
\begin{proof}
Suppose not, so $|D_{00}| > n^2 - (2n - 43)$.  By Lemma \ref{lem: c1}, no member of $D_{00}$ has a private neighbor in $N_{00}$.  Since $D_{10}$ and $D_{01}$ are nonempty, at least $2n - 3$ vertices of $N_{11}$ are in the neighborhood of $D_{10} \cup D_{01}$.  Thus, at most $n^2 - (2n - 3)$ members of $D_{00}$ have a private neighbor in $N_{11}$.  Observe that at most $10$ rows (resp. columns) of $N_{00}$ contain at most $2$ members of $D_{00}$, as $|N_{00} \cut D_{00}| \leq 10n < 11(n - 2)$ for $n \geq 71$.  Thus at most $20$ members of $D_{00}$ have a private neighbor in $N_{01}$, and similarly at most $20$ members of $D_{00}$ have a private neighbor in $N_{10}$.  Therefore there at most $n^2 - (2n - 3) + 40 = n^2 - (2n - 43)$ vertices in $D_{00}$.
\end{proof}

\begin{lem}\label{lem: c4}
Suppose $D$ is a minimal dominating set of $G_n \times G_n$ of size greater than $n^2$ and the conditions of Case (\ref{c}) hold. Then the elements of $D_{01}$ (resp. $D_{10}$) span at most $10$ columns (resp. rows).  Moreover, the private neighbors of $D$ in $N_{01}$ (resp. $N_{10}$) span at most $10$ rows (resp. columns). Hence there are at most $19$ private neighbors of $D_{01}$ (resp. $D_{10}$) in $N_{10}$ (resp. $N_{01}$).
\end{lem}
\begin{proof}
By Lemma \ref{lem: c1}, at least $n^2 - 10n$ vertices of $N_{11}$ are private neighbors of elements of $D_{00}$.  Thus at most $10n$ vertices of $N_{11}$ are adjacent to vertices of $D_{01}$.  Given $v_1,\dots,v_k \in D_{01}$, each in distinct columns, observe that $k(n-1)$ vertices of $N_{11}$ are in the neighborhood of $\{v_1,\dots,v_k\}$.  Thus $k(n-1) \leq 10n$, and as $n \geq 71$, we have $k \leq 10$. A similar counting argument shows the vertices of $D_{10}$ span at most 10 rows.  

As discussed in the proof of the previous lemma, at most $10$ rows of $N_{00}$ contain at most $2$ vertices of $D_{00}$.  Thus at most $10$ rows of $N_{01}$, namely those corresponding to a row of $N_{00}$ containing at most $2$ vertices of $D_{00}$, contain a vertex which is not adjacent to two vertices of $D_{00}$.  Thus the private neighbors of $D$ in $N_{01}$ span at most $10$ rows.  An analogous argument shows that the private neighbors of $D$ in $N_{10}$ span at most $10$ columns.

It remains to prove the last claim.  Given $v \in D_{01}$, note that its only neighbor of in $N_{10}$ is its corresponding vertex. By the previous claims, there are at most $10$ rows and $10$ columns in which a private neighbor in $N_{10}$ of a vertex $v \in D_{01}$ could lie. Moreover, since $N_{10}$ is nonempty, there is at most $1$ row and $1$ column of $N_{10}$ is not in the closed neighborhood of $D_{10}$.  Combining these restrictions shows that at most $2\cdot 10 - 1 = 19$ vertices of $N_{10}$ can be private neighbors of vertices in $D_{01}$. This argument proceed analogously swapping the roles of $D_{10}$ and $D_{01}$ as well as $N_{10}$ and $N_{01}$.
\end{proof}

\begin{proof}[Proof of Theorem \ref{thm: upper dom}]
Let $D$ be a minimal dominating set of $G_n \times G_n$ of size greater than $n^2$.  We proceed by considering the four cases discussed above, reaching a contradiction in each: 
\begin{enumerate}
    \item[{(\ref{a})}:]  Assume that the classes $D_{00}$ and $D_{11}$ are nonempty and the remaining classes are empty.  Note that there are at most $2n$ private neighbors of $D$ in each of $N_{00}$ and $N_{11}$, since the closed neighborhood any vertex in $N_{00}$ (resp. $N_{11}$) contains all but $2n-1$ vertices of $N_{00}$ (resp. $N_{11}$).  We next claim that there are at most $2n$ private neighbors of $D_{00}$ in $N_{01}$. Observe that any two vertices in the same column of $N_{01}$ have neighborhoods in $N_{00}$ whose setwise difference contains precisely one vertex.  That is, if any three vertices of $D_{00}$ are adjacent to a vertex $v \in N_{01}$, then every vertex in the same row as $v$ in $N_{01}$ is adjacent to at least two vertices of $D_{00}$.  Hence there are at most two private neighbors of $D_{00}$ in any row of $N_{01}$.  Analogous arguments show that there are at most $2n$ private neighbors of $D_{00}$ in $N_{10}$ and at most $4n$ private neighbors of $D_{11}$ in $N_{01} \cup N_{10}$.   Therefore there are at most $12n$ private neighbors of $D$ in total.  Since we assume $n \geq 71$, we have $|D| \leq 12n < n^2$.
    \item[{(\ref{b})}:] Assume that the classes $D_{00}$ and $D_{01}$ are nonempty and the remaining classes are empty.   Let 
\begin{align*}
    W(b_0) &= \{(c,d) \in \Z_2 \times \Z_n : \exists a \in \Z_2 \text{ such that } (a,b_0,c,d) \in D\}\\
    &= \{(c,d) \in \Z_2 \times \Z_n : (0,b_0,c,d) \in D\}\,.
\end{align*}
The set $W(b_0)$ can be viewed as the projection of the $b_0$-columns of $D$ onto the last two coordinates.  Observe that if $(0,b,c,d) \in D$ is a neighbor of the vertex $(1,b_0,c_0,d_0)$, then $b = b_0$.  Since every vertex of the latter form must be in the neighborhood of $D$, $W(b_0)$ must be a total dominating set in $G_n$ (using the correspondence of the vertex set with $\Z_2 \times \Z_n$).  Observe then that 
$$\bigcup_{b_0 \in \Z_n} \{(0,b_0)\} \times W(b_0) \subseteq D$$
is a dominating set of $G_n \times G_n$, hence is equal to $D$ by our assumption that $D$ is a minimal dominating set.  Moreover, by the minimality of $D$, each $W(b_0)$ should be a minimal dominating set of $G_n$.  It is straightforward to check that a minimal total dominating set in $G_n$ has size $2$, $4$, or $n$.  Thus, 
$$|D| = \sum_{b_0 = 0}^{n-1} |W(b_0)| \leq \sum_{b_0 = 0}^{n-1} n = n^2.$$
    \item[{(\ref{c})}:]
    Assume that the classes $D_{00}$, $D_{10}$, and $D_{01}$ are nonempty and that $D_{11}$ is empty. 
 Without loss of generality, by Lemma \ref{lem: c3}, we can assume $|D_{10}| \geq n - 21$.  Lemma \ref{lem: c1} implies that no member of $D_{10}$ has a private neighbor in $N_{00}$.  Moreover, $D_{10}$ spans at most $10$ columns by Lemma \ref{lem: c4} and each column of $D_{10}$ has at most $2$ vertices with private neighbors in $N_{11}$.  Hence there are at most $20$ private neighbors of $D_{10}$ in $N_{11}$.  By Lemma \ref{lem: c4}, the vertices of $D_{10}$ have at most $19$ private neighbors in $N_{01}$.  Therefore, there must be at least $(n - 21) - 20 - 19 = n - 60$ private neighbors of $D_{10}$ in $N_{10}$.  

Since $n \geq 71$, then there are at least $11$ members of $D_{10}$ with a private neighbor in $N_{10}$.  Note that all such vertices must lie either in the same row or the same column by Observation \ref{obs: adj}.  However, the first claim of Lemma \ref{lem: c4} implies that these vertices cannot lie in the same row, lest they span $11$ columns.  The second claim of Lemma \ref{lem: c4} implies that they cannot lie in the same column, lest their private neighbors (namely, the vertices themselves) span $11$ rows.  Therefore no such minimal dominating set $D$ can exist.
\item[{(\ref{d})}:] Assume that all four of the classes $D_{ij}$ are nonempty.  Fix a vertex $v_{ij} \in D_{ij}$ for each $(i,j) \in \Z_2 \times \Z_2$.  Since all but $2n-1$ vertices of $D_{ij}$ are in the closed neighborhood of $v_{ij}$, we have $V(H) \setminus N[\{v_{00},v_{01},v_{10}, v_{11}\}] \leq 8n-4$.  Thus by the minimality of $D$ and our assumption that $n \geq 71$, we have $|D| \leq 8n \leq n^2$. \qedhere 
\end{enumerate}
\end{proof}

\section{Further Directions}\label{sec: further}

As described in Lemma \ref{lem: match dom prod com}, the paired domination number exhibits additive behavior on certain direct products of complete graphs.  In particular, this result implies that there are infinitely many graphs $G$ and $H$ such that $\gamma_{\mathrm{pr}}(G \times H) = \gamma_{\mathrm{pr}}(G) + \gamma_{\mathrm{pr}}(H) - 1$.  For the domination number, Meki{\v s} \cite[Theorem 3.1]{Mek} proved the additive lower bound $\gamma(G \times H) \geq \gamma(G) + \gamma(H) - 1$.  Using the bound $\gamma_{\mathrm{pr}}(G) \leq 2\gamma(G)$ (see \cite[Theorem 9]{HS}), we obtain $\gamma_{\mathrm{pr}}(G \times H) \geq \frac{1}{2}\left( \gamma(G) + \gamma(H)\right) - 1$.  This motivates the following question,

\begin{ques}
What is the largest constant $c \in \left[\frac{1}{2}, 1\right]$ such that 
$$\gamma_{\mathrm{pr}}(G \times H) \geq c\left( \gamma_{\mathrm{pr}}(G) + \gamma_{\mathrm{pr}}(H)\right) - 1$$
holds for all graphs $G$ and $H$?
\end{ques}

Corollary \ref{cor: tree prod} asserts that the inequality $\gamma_{\mathrm{pr}}(T_1 \times T_2) \geq \frac{1}{2}\gamma_{\mathrm{pr}}(T_1)\gamma_{\mathrm{pr}}(T_2)$ holds for any trees $T_1$ and $T_2$.  As discussed in Section \ref{Paired Dom}, Bre{\v s}ar, Klav{\v z}ar, and Rall \cite{BKR} showed that the constant factor $\frac{1}{2}$ is tight here by examining these parameters on the subdivided stars. It is yet not known whether equality can be achieved, but the author conjectures that this is not the case.

\begin{conj}
For any trees $T_1$ and $T_2$ of order at least $2$, we have
$$\gamma_{\mathrm{pr}}(T_1 \times T_2) > \frac{1}{2} \gamma_{\mathrm{pr}}(T_1)\gamma_{\mathrm{pr}}(T_2)\,.$$
\end{conj}

It is established in Theorem \ref{thm: prod subgraph} that any graphs can appear as induced subgraphs of a pair of graphs $G,H$ satisfying $\gamma_{\mathrm{pr}}(G \times H) \geq \frac{1}{2}\gamma_{\mathrm{pr}}(G)\gamma_{\mathrm{pr}}(H)$.  It is possible that the constant factor of $\frac{1}{2}$ can be improved in this statement.  

\begin{ques}
Does there exist $c > \frac{1}{2}$ (and if so, what is the supremum of such $c$) such that for any graphs $G$ and $H$, there exists graphs $G'$ and $H'$ containing $G$ and $H$, respectively, as induced subgraphs such that 
$$\gamma_{\mathrm{pr}}(G' \times H') \geq c \gamma_{\mathrm{pr}}(G')\gamma_{\mathrm{pr}}(H')\,?$$
\end{ques}

By Theorem \ref{thm: intro pair prod}, no multiplicative lower bound holds for the paired domination numbers of direct product graphs.  However, one may be able to characterize the classes of graphs satisfying a particular multiplicative lower bound.  
\begin{prob}
For any fixed $c > 0$, characterize the pairs of graphs $G$ and $H$ satisfying 
$$\gamma_{\mathrm{pr}}(G \times H) > c\gamma_{\mathrm{pr}}(G)\gamma_{\mathrm{pr}}(H)\,.$$
\end{prob}
In particular, it may be interesting to focus on the cases $c = 1$ or $c = \frac{1}{2}$.  One could also study the symmetric version of this question, where $G = H$.

It may also be interesting to characterize the graphs satisfying the equality in Theorem \ref{thm: upper dom}.  Thus far, no families of examples with unbounded upper domination numbers are known other than the $2 \times n$ rook graphs. 

\begin{prob}
Construct other large families of graphs $G$ and $H$ with arbitrarily large upper domination numbers satisfying $\Gamma(G \times H) = \Gamma(G)\Gamma(H)\,.$
\end{prob}

\section{Acknowledgements}
The author was supported by a Marshall scholarship and a St.\ John's College Benefactors' scholarship.

\end{document}